\documentclass[12pt]{article}
\usepackage{amsmath}
\usepackage{amssymb}
\usepackage{amsfonts}
\usepackage{amsthm}
\usepackage{enumerate}
\usepackage{longtable}
\usepackage{hyperref}

\def\init{\setcounter{equation}{0}}

\newtheorem{theorem}{Theorem}[section]
\newtheorem{proposition}[theorem]{Proposition}

\newtheorem{definition}[theorem]{Definition}

\def\Li{{\operatorname{Li}}}

\def\d{{\operatorname{d}}}

\def\mod{\hbox{mod}\,}

\numberwithin{equation}{section}
\newenvironment{acknowledgements}{\noindent{\bf Acknowledgements}\bigskip}{}

\begin{document}

\title{An extension of Boyd's $p$-adic algorithm for the harmonic series}
\author{Mathew D. Rogers\\
        \small{\textit{Department of Mathematics, University of British
        Columbia}}\\
        \small{\textit{Vancouver, BC, V6T-1Z2, Canada}}\\
        \small{\textbf{email:} matrogers@math.ubc.ca}}

\maketitle

\abstract{In this paper we will extend a $p$-adic algorithm of Boyd
in order to study the size of the set:
\[J_p(y)=\left\{n :\sum_{j=1}^{n}\frac{y^j}{j}\equiv 0\left(\mod
p\right)\right\}.\]  Suppose that $p$ is one of the first $100$ odd
primes and $y\in\{1,2,\dots,p-1\}$, then our calculations prove that
$|J_p(y)|<\infty$ in $24240$ out of $24578$ possible cases. Among
other results we show that $|J_{13}(9)|=18763$. The paper concludes
by discussing some possible applications of our method to sums
involving Fibonacci numbers.}


\section{Introduction}
\label{intro} \init

    The goal of this work is to extend an algorithm from David Boyd's paper,
``A $p$-adic study of the partial sums of the harmonic series". In
particular, Boyd developed an algorithm which enabled him to
calculate complete solution sets of the congruence
\begin{equation}\label{boyds initial equation}
\sum_{j=1}^{n}\frac{1}{j}\equiv 0\left(\mod p\right),
\end{equation}
for $497$ of the first $500$ primes (his calculations did not finish
for $p\in\{83,127,397\}$, and he accidentally omitted $p=509$ from
his table). Boyd's computations strongly supported the hypothesis he
drew from probabilistic models, namely that Eq. \eqref{boyds initial
equation} should only have a finite number of solutions in $n$ for
every prime $p$ \cite{Bo}.

In this paper we will consider the much larger class of sums defined
by
\begin{equation}\label{sum yj/j =0 mod p}
G_n(y):=\sum_{j=1}^{n}\frac{y^j}{j}.
\end{equation}
We will extend Boyd's computational method to find complete
solutions sets of the congruence $G_n(y)\equiv 0 \left(\mod
p\right)$ for every prime $p<550$, and for every $p$-adic integer
$y$ with $\nu_p(y)=0$, with $338$ exceptions where our calculations
did not finish. Our computations rely upon results derived in
Theorems \ref{recurrence for Gn(y)} and \ref{Gn(y^p) congruence
theorem}. Table $1$ lists the number of solutions of the congruence
$G_n(y)\equiv 0\left(\mod p\right)$ for $p<50$ and
$y\in\{1,2,\dots,p-1\}$.

This topic is also closely related to several unsolved problems in
classical number theory.  Boyd mentioned the connection between Eq.
\eqref{boyds initial equation}, Bernoulli numbers, and regular
primes. It turns out that the zeros of $G_n(y)\left(\mod p\right)$
are also related to the Wieferich primes.  The elementary congruence
$G_{p-1}(2)\equiv \frac{2^p-2}{p} \left(\mod p\right)$, implies that
$G_{p-1}(2)\equiv 0\left(\mod p\right)$ holds if and only if
$2^{p-1}\equiv 1\left(\mod p^2\right)$.  Primes that satisfy
$2^{p-1}\equiv 1\left(\mod p^2\right)$ are called Weiferich primes,
and the only known examples are $p=1093$ and $p=3511$ (see \cite{Es}
or \cite{Ri1}).  Despite the fact that just two Weiferich primes
have been discovered, the proof that infinitely
\textit{non}-Wieferich primes exist depends upon the ABC conjecture.
Thus it would be extremely interesting to find an unconditional
proof that ``$|\{n:G_n(2)\equiv 0 \left(\mod p\right)\}|=0$" holds
infinitely often, since this condition implies that $p$ is not a
Wieferich prime.  Our computations (and heuristics) seem to suggest
that approximately $36 \%$ of primes satisfy this last condition. In
general we believe that $|\{n:G_n(y)\equiv 0\left(\mod
p\right)|<\infty$ whenever $\nu_p(y)=0$, and our calculations have
confirmed that this is true in at least $24240$ out of $24578$ cases
for $p<550$.

Finally, we will conclude the paper with a brief discussion of some
further possible extensions of our method.  For example, we will
show that Boyd's algorithm can be applied to study the congruence
\begin{equation}\label{sum fibonaccij/j =0 mod p}
\sum_{j=1}^{n}\frac{F_j}{j}\equiv 0 \left(\mod p\right),
\end{equation}
when the $F_j$'s are Fibonacci numbers.  It seems likely that a
large number of interesting congruences similar to Eq. \eqref{sum
fibonaccij/j =0 mod p} can also be studied using our method.

\section{Elementary properties of $J_N(y)$ and $G_n(y)$}
\label{algorithm} \init

Although the primary goal of this paper is to determine complete
solution sets of the congruence $G_n(y)\equiv 0\left(\mod p\right)$,
we can easily consider the more general case of $J_N(y)$ when
$N\in\mathbb{Z}$.
\begin{definition}
Let $J_N(y)$ be defined by
\begin{equation}
J_N(y)=\left\{n: G_n(y)\equiv 0\left(\mod N\right)\right\}.
\end{equation}
\end{definition}
Before performing any calculations we will use elementary number
theory to narrow the scope of our investigation.  First notice
that $G_n(y)\equiv 0\left(\mod N\right)$ if and only if
$G_n(y)\equiv 0\left(\mod p^{s}\right)$ for every prime power
$p^s$ dividing $N$. Therefore it is obvious that
\begin{equation}\label{J intersection equation}
J_{p_1^{s_1}\dots p_n^{s_n}}(y)=\bigcap_{i=1}^{n}J_{p_i^{s_i}}(y),
\end{equation}
whenever the $p_i$'s are distinct primes.  Likewise, it is clear
that for any prime $p$ we must have a sequence of inclusions:
\begin{equation*}
J_p(y)\supseteq J_{p^2}(y)\supseteq J_{p^3}(y)\dots
\end{equation*}
In Section \ref{computations} we will show that easiest way to
calculate $J_{p^{s}}(y)$ is to first determine $J_p(y)$, and then to
check whether or not $G_n\left(y\right)\equiv0 \left(\mod
p^s\right)$ for every $n\in J_p(y)$.  Surprisingly, it is often more
difficult to determine $J_{p^s}(y)$ than $J_p(y)$.

Now we will discuss which values of $y\in\mathbb{Q}$ need to be
considered. Elementary number theory shows that calculating
$J_{p^s}(y)$ is easy if $\nu_p(y)>0$ because almost all of the
terms in $G_n(y)$ vanish modulo $p^s$.  Likewise, it is clear that
$J_{p^s}(y)=\emptyset$ if $\nu_p(y)<0$.  Finally, when
$\nu_p(y)=0$ we can appeal to the following proposition:

\begin{proposition}\label{Gn(y) and J_N(y) reduction mod p^s theorem} Suppose that $\nu_p(y)=0$.  If
$y\equiv \bar{y}\left(\mod p^{s}\right)$, then
\begin{equation}\label{Gn(y) general prime power congruence}
G_n(y)\equiv G_n(\bar{y})\left(\mod p^s\right),
\end{equation}
and it follows that
\begin{equation}\label{Jp^alpha equality}
J_{p^{s}}(y)=J_{p^{s}}(\bar{y}).
\end{equation}
\end{proposition}
\begin{proof} Notice that Eq. \eqref{Gn(y) general prime power congruence}
is equivalent to the congruence
\begin{equation*}
\sum_{j=1}^{n}\frac{\bar{y}^{j}-y^{j}}{j}\equiv 0\left(\mod
p^{s}\right),
\end{equation*}
which follows trivially from the claim that for any $j\ge 1$
\begin{equation*}
\frac{\bar{y}^{j}-y^{j}}{j}\equiv 0\left(\mod p^{s}\right).
\end{equation*}
 To prove this claim, suppose $j=j'p^\gamma$
where $(j',p)=1$. Since $\nu_p(y)=0$, rearrangement shows that
\begin{equation*}
    \left(y/\bar{y}\right)^{p^{\gamma}j'}-1\equiv
    0\left(\mod p^{\gamma+s}\right).
\end{equation*}
Since $y/\bar{y}\equiv1\left(\mod p^s\right)$, there exists a
$p$-adic integer $\beta$ such that $y/\bar{y}=1+\beta p^{s}$, and
therefore
\begin{equation*}
    \left(1+p^{s}\beta\right)^{p^{\gamma}j'}-1\equiv
    0\left(\mod p^{\gamma+s}\right).
\end{equation*}
Induction on $\gamma$ verifies this last equality.$\blacksquare$
\end{proof}

It follows immediately from Proposition \ref{Gn(y) and J_N(y)
reduction mod p^s theorem} that $
J_p(y)\in\left\{J_p(1),J_p(2),\dots,J_p(p-1)\right\}$ whenever
$\nu_p(y)=0$.  Recall that Boyd was able to calculate $|J_p(1)|$ for
$p<550$ because the set $J_p(1)$ possesses a tree structure. In
particular, an integer $n$ can only belong to $J_p(1)$ if the
integer part of $n/p$ also belongs to $J_p(1)$. If the harmonic
series is defined by
\begin{equation}
H_n:=G_n(1)=\sum_{j=1}^n\frac{1}{j},
\end{equation}
then for $k\in\{0,1,\dots,p-1\}$ and $n\ge 0$
\begin{equation}\label{harmonic series recursion}
 H_{pn+k}\equiv\frac{H_n}{p}+H_k \left(\mod p\right).
\end{equation}
It follows from Eq. \eqref{harmonic series recursion} that $H_{p
n+k}\equiv 0 \left(\mod p\right)$ can only hold if $H_n\equiv
0\left(\mod p\right)$ holds as well, thus the set $J_p(1)$ inherits
its tree structure from Eq. \eqref{harmonic series recursion}.  Boyd
calculated $J_p(1)$ using the following algorithm:

\medskip
\noindent\textbf{Boyd's Algorithm:}

\begin{center}
{\small \framebox[6in]{
\begin{minipage}[t]{5.0in}
      First check if $H_{k}\equiv 0\left(\mod p\right)$ for
every $k\in\{1,\dots,p-1\}$. For each $k$ that satisfies
$H_k\equiv0\left(\mod p\right)$, check if any elements of $\{H_{p
k}, H_{p k+1}, H_{p k+p-1}\}$ also vanish modulo $p$. Iterate this
argument whenever $H_{pk+j}\equiv 0\left(\mod p\right)$ for some
$j\in\{0,1,\dots,p-1\}$.  The algorithm terminates in a finite
amount of time if $|J_p(1)|<\infty$.  
\end{minipage}
} } \end{center}

    In order to apply Boyd's algorithm to the problem of calculating $J_p(y)$, we will
first need to prove that the set $J_p(y)$ has a tree structure for
$y\in\{2,3,\dots p-1\}$.  In particular the the next theorem
proves that $G_n(y)$ satisfies a recurrence relation which reduces
to Eq. \eqref{harmonic series recursion} whenever $y=1$.

\begin{theorem}\label{recurrence for Gn(y)} Suppose that $\nu_p(y)=0$
and $k\in\{0,1,\dots,p-1\}$, then
\begin{equation}\label{Gn recurrence}
G_{pn+k}\left(y^p\right)\equiv\frac{G_{n}\left(y^p\right)}{p}+y^n
G_k(y)+\left(\frac{y^n-1}{y-1}\right)G_{p-1}(y)\left(\mod p\right).
\end{equation}
\end{theorem}
\begin{proof}
    First observe that for $k\in\{0,1,\dots,p-1\}$
\begin{align*}
G_{pn+k}\left(y^{p}\right)
&=\sum_{j=1}^{n}\frac{y^{p^2 j}}{p
j}+\sum_{j=1}^{p-1}\sum_{r=0}^{n-1}\frac{y^{p(j+pr)}}{j+pr}+\sum_{j=1}^{k}\frac{y^{p(j+pn)}}{j+pn}.
\end{align*}
Reducing modulo $p$, this becomes
\begin{align*}
G_{p n+k}\left(y^{p}\right)
&\equiv\frac{G_{n}\left(y^{p^{2}}\right)}{p}+\left(\frac{y^n-1}{y-1}\right)G_{p-1}(y)+y^n
G_k(y)\left(\mod
p\right)\notag\\
&\equiv\frac{G_{n}\left(y^{p}\right)}{p}+\frac{G_{n}\left(y^{p^{2}}\right)-G_{n}\left(y^{p}\right)}{p}+\left(\frac{y^n-1}{y-1}\right)G_{p-1}(y)+y^n
G_k(y)\left(\mod p\right)\label{Gn intermed congruence}.
\end{align*}
Since Eq. \eqref{Gn(y) general prime power congruence} shows that
$G_{n}\left(y^{p^{2}}\right)- G_{n}\left(y^{p}\right)\equiv
0\left(\mod p^2\right)$, it is easy to check that $
\left(G_{n}\left(y^{p^{2}}\right)-
G_{n}\left(y^{p}\right)\right)/p\equiv 0\left(\mod p\right)$, and
the theorem follows.$\blacksquare$
\end{proof}

    Although Eq. \eqref{Gn recurrence} does not apply
to $G_n(y)$ directly, we can still use it to calculate all of the
integers belonging to $J_p(y)$. In particular, Eq. \eqref{Gn(y)
general prime power congruence} shows that whenever $\nu_p(y)=0$:
\begin{equation}\label{Jp(y^p)=Jp(y)}
J_p\left(y^p\right)=J_p(y).
\end{equation}
Since Eq. \eqref{Gn recurrence} shows that $J_p\left(y^p\right)$
has a tree structure, we can calculate $J_p\left(y^p\right)$ by
simply modifying Boyd's algorithm to replace $H_n$ with
$G_n\left(y^p\right)$.

\subsection{A $p$-adic expansion for $G_{p
n}\left(y^p\right)-\frac{G_{n}\left(y^p\right)}{p}$} \label{p adic
series section} \init


Although in principle $J_p(y)$ can be calculated by combining Boyd's
algorithm with equations \eqref{recurrence for Gn(y)} and
\eqref{Jp(y^p)=Jp(y)}, such a naive approach will rapidly exhaust a
computer's memory. Boyd encountered a similar problem since both the
numerator and denominator of $H_n$ grow exponentially as functions
of $n$. In order to calculate $H_n$ efficiently, he used a $p$-adic
series for $H_{p n}-H_n/p$.  Boyd showed that for $s\ge 2$ there
exists a polynomial, $X(n)$, of degree $s-1$ such that
\begin{equation*}
H_{p n}-\frac{H_n}{p}\equiv X(n)\left(\mod p^s\right).
\end{equation*}
Notice that the right-hand side of this last congruence is easy to
calculate even for extraordinarily large values of $n$. Thus if
$H_n\left(\mod p^{s+1}\right)$ is known, the value of $H_{p
n}\left(\mod p^s\right)$ follows easily from rearranging the
congruence to obtain
\begin{align*}
H_{p n}\equiv& X(n)+\frac{H_n}{p}\left(\mod p^s\right).
\end{align*}
This argument can be iterated $s-1$ times.  For example, it is
clear that
\begin{equation*}
\begin{split}
H_{p^2 n}&\equiv X(p n)+\frac{H_{p n}}{p}\left(\mod
p^{s-1}\right)\\
&\equiv X(p n)+\frac{H_{p n}\left(\mod p^s\right)}{p}\left(\mod
p^{s-1}\right)\\
&\equiv X(p n)+\frac{X(n)}{p}+\frac{H_{n}}{p^2}\left(\mod
p^{s-1}\right)
\end{split}
\end{equation*}
Although every iteration of this argument causes the loss of one
digit of $p$-adic precision, we can still compute the value of
$H_n\left(\mod p\right)$ for every integer $n\in J_p(1)$ by simply
starting with a large enough initial value of $s$ (assuming of
course that $|J_p(1)|<\infty$).

Therefore it is obvious that we will need to derive a $p$-adic
expansion for $G_{pn}\left(y^p\right)-G_n\left(y^p\right)/p$. Recall
that Boyd determined his $p$-adic expansion for $H_{p n}-H_{n}/p$ by
first calculating the value of the function for $n\in\{1,\dots s\}$,
and then by multiplying those values times the inverse of an
$s$-dimensional Vandermonde matrix. Unfortunately Boyd's approach
usually fails in this case, as the numerator of $G_{p
n}\left(y^p\right)-G_{n}\left(y^p\right)/p$ grows far too quickly to
allow for direct calculation. As a result, we will avoid those
computations altogether by finding explicit formulas for the
$p$-adic series coefficients. The first step is to split the
function into two pieces:
\begin{equation}\label{split up the gn(y^p) sum}
G_{p
n}\left(y^p\right)-\frac{G_{n}\left(y^p\right)}{p}=\sum_{\substack{j=1\\
(j,p)=1}}^{p n}\frac{y^{p j}}{j}+\sum_{j=1}^{n}\frac{y^{p^2
j}-y^{p j}}{p j}.
\end{equation}
In the next theorem we prove two useful formulas for calculating Eq.
\eqref{split up the gn(y^p) sum}.

\begin{theorem}\label{Gn(y^p) congruence theorem}
Suppose that $p>2$, $\nu_p(y)=0$, and $y\not= 1$, then we have the
following expansions:
\begin{align}
\sum_{\substack{j=1\\ (j,p)=1}}^{pn}\frac{y^{p j}}{j}\equiv &
A_{0}(s)-y^{p^2 n}\sum_{i=0}^{s-1} A_i(s)n^{i} \left(\mod
p^s\right),\label{Gn(y^p) relatively prime part}\\
\sum_{j=1}^{n}\frac{y^{p^2 j}-y^{p j}}{p j}\equiv & N_0(s)-y^{p
n}\sum_{i=0}^{s-1}N_i(s)n^i\left(\mod p^s\right).\label{Gn(y^p)
not relatively prime part}
\end{align}
We can calculate $A_i(s)$ and $N_i(s)$ using
\begin{align}
A_{i}(s)\equiv & \sum_{m=i}^{s-1}(-1)^m p^m{m\choose
i}\left(\sum_{k=1}^{p-1}\frac{y^{p
k}}{k^{m+1}}\right)\left(\sum_{k=0}^{m-i}k!\mathfrak{S}_{m-i}^{(k)}\frac{y^{
p^2 k}}{\left(1-y^{p^2}\right)^{k+1}}\right)\left(\mod
p^s\right),\label{Ai(s) definition}\\
N_i(s)\equiv&-p^{2i+1}\frac{z^{i+1}}{(i+1)!}\notag\\
&+\sum_{m=i}^{s-2}p^{2m+1}\frac{ z^{m+1}}{(m+1)!}{m\choose
i}\left(\sum_{k=0}^{m-i}k!\mathfrak{S}_{m-i}^{(k)}\frac{y^{p
k}}{\left(1-y^p\right)^{k+1}}\right)\left(\mod
p^s\right),\label{Ni(s) definition}
\end{align}
where $z$ is a $p$-adic integer defined by
\begin{equation}\label{z def}
z\equiv\frac{1}{p^2}\sum_{j=1}^{s-1}\frac{(-1)^{j+1}}{j}\left(y^{p(p-1)}-1\right)^j\left(\mod
p^s\right),
\end{equation}
and the $\mathfrak{S}_{m}^{(k)}$'s are Stirling numbers of the
second kind \cite{Gr}.
\end{theorem}
\begin{proof} We will prove equations \eqref{Gn(y^p) relatively prime part} and \eqref{Ai(s) definition} first.
Rearranging the sum shows that
\begin{equation*}
\sum_{\substack{j=1\\ (j,p)=1}}^{pn}\frac{y^{p
j}}{j}=\sum_{k=1}^{p-1}\frac{y^{pk}}{k}+\sum_{k=1}^{p-1}\sum_{r=1}^{n-1}\frac{y^{p(k+pr)}}{k+p
r}.
\end{equation*}
Reducing modulo $p^s$ this becomes
\begin{equation*}
\begin{split}
\sum_{\substack{j=1\\ (j,p)=1}}^{pn}\frac{y^{p j}}{j}
&\equiv\sum_{k=1}^{p-1}\frac{y^{pk}}{k}+\sum_{k=1}^{p-1}\frac{y^{p
k}}{k}\sum_{r=1}^{n-1}y^{p^2
r}\left(\frac{1-\left(\frac{-pr}{k}\right)^s}{1-\left(\frac{-p
 r}{k}\right)}\right)\left(\mod p^s\right).
\end{split}
\end{equation*}
Employing a geometric series yields
\begin{equation}\label{Gn(y^P) relatively prime part intermed}
\begin{split}
\sum_{\substack{j=1\\ (j,p)=1}}^{pn}\frac{y^{p j}}{j}
&\equiv\left(\sum_{k=1}^{p-1}\frac{y^{pk}}{k}\right)+\sum_{m=0}^{s-1}(-1)^m
p^m \left(\sum_{k=1}^{p-1}\frac{y^{p
k}}{k^{m+1}}\right)\left(\sum_{r=1}^{n-1}r^m y^{p^2
r}\right)\left(\mod p^s\right).
\end{split}
\end{equation}
Now we will assume that $y\not=1$, then by Eq. \eqref{truncated
polylog summation formula}
\begin{equation*}
\sum_{r=1}^{n-1}r^m y^{p^2 r}=-n^m y^{p^2
n}+\Li_{-m}\left(y^{p^2}\right)-y^{p^2 n}\sum_{j=0}^{m}{m\choose
j}n^j \Li_{-(m-j)}\left(y^{p^2}\right).
\end{equation*}
Substituting this result into Eq. \eqref{Gn(y^P) relatively prime
part intermed} yields
\begin{equation*}
\sum_{\substack{j=1\\ (j,p)=1}}^{pn}\frac{y^{p j}}{j}\equiv
A_0(s)-y^{p^2 n}\sum_{i=0}^{s-1}A_i(s)n^i\left(\mod p^s\right),
\end{equation*}
where
\begin{equation*}
A_i(s)=(-1)^i p^i\left(\sum_{k=1}^{p-1}\frac{y^{p
k}}{k^{i+1}}\right)+\sum_{m=i}^{s-1}(-1)^m p^m{m\choose
i}\left(\sum_{k=1}^{p-1}\frac{y^{pk}}{k^{m+1}}\right)\Li_{-(m-i)}\left(y^{p^2}\right).
\end{equation*}
Eq. \eqref{Ai(s) definition} follows from combining this
definition of $A_i(s)$ with Eq. \eqref{summation for finite
polylog}.

Now we will prove equations \eqref{Gn(y^p) not relatively prime
part} and \eqref{Ni(s) definition}.  First observe that if
$\alpha=\left(y^{p(p-1)}-1\right)/p^2$, then
\begin{equation*}
\begin{split}
\sum_{j=1}^{n}\frac{y^{p^2 j}-y^{p j}}{p j}
&=\sum_{j=1}^{n}\frac{y^{p j}}{p j}\left(\left(1+\alpha
p^2\right)^{j}-1\right).
\end{split}
\end{equation*}
Applying the binomial formula shows that
\begin{align*}
\sum_{j=1}^{n}\frac{y^{p^2 j}-y^{p j}}{p j}
&=\sum_{j=1}^{n}\frac{y^{p j}}{p j}\sum_{k=1}^{j}{j\choose
k}\alpha^k p^{2k}\notag\\
&=\sum_{k=1}^{n}\frac{\alpha^k
p^{2k-1}}{k!}\sum_{j=1}^{n}\frac{y^{p
j}}{j}\left(j(j-1)\dots(j-k+1)\right).
\end{align*}
Reducing this last equation modulo $p^s$ is easy. The nested sum
is an integer, and since $\nu_p(y)=0$ elementary number theory
shows that $\nu_p(\frac{\alpha^k p^{2k-1}}{k!})\ge k$ for all $k$.
Therefore we can truncate the right-hand sum after the first $s-1$
terms to obtain
\begin{align}\label{first sum intermediate}
\sum_{j=1}^{n}\frac{y^{p^2 j}-y^{p j}}{p j}
\equiv\sum_{k=1}^{s-1}\frac{\alpha^k
p^{2k-1}}{k!}\sum_{j=1}^{n}\frac{y^{p
j}}{j}\left(j(j-1)\dots(j-k+1)\right)\left(\mod p^s\right).
\end{align}
We will simplify Eq. \eqref{first sum intermediate} by using
properties of the Stirling numbers of the first kind.  Recall that
for $k\ge 1$ the Stirling numbers of the first kind have the
generating function
\begin{equation*}
x(x-1)\dots(x-k+1)=\sum_{m=1}^{k}S_k^{(m)}x^m.
\end{equation*}
Substituting this definition into the nested sum in Eq.
\eqref{first sum intermediate} yields
\begin{equation*}
\begin{split}
\sum_{j=1}^{n}\frac{y^{p j}}{j}\left(j(j-1)\dots(j-k+1)\right)
=\sum_{j=1}^{n}\frac{y^{p j}}{j}\sum_{m=1}^{k}S_k^{(m)}j^m
=\sum_{m=0}^{k-1}S_k^{(m+1)}\sum_{j=1}^{n}j^{m}y^{p j},\\
\end{split}
\end{equation*}
and therefore Eq. \eqref{first sum intermediate} becomes
\begin{equation}\label{first sum intermediate part two}
\sum_{j=1}^{n}\frac{y^{p^2 j}-y^{p j}}{p j}\equiv
\sum_{m=0}^{s-2}\left(\sum_{k=m+1}^{s-1}S_k^{(m+1)}\frac{\alpha^k
p^{2k-1}}{k!}\right)\left(\sum_{j=1}^{n}j^m y^{p
j}\right)\left(\mod p^s\right).
\end{equation}

    Now we will use a second power series identity
for Stirling numbers.  It is well known \cite{Gr} that if $|x|<1$
\begin{equation*}
\frac{1}{m!}\left(\sum_{k=1}^{\infty}\frac{(-1)^{k+1}}{k}x^k\right)^m=\sum_{k=m}^{\infty}S_k^{(m)}\frac{x^k}{k!}.
\end{equation*}
It follows that for some integer-valued polynomial $Q(x)$
\begin{equation*}
\frac{p^{m}}{m!}\left(\sum_{k=1}^{s-1}(-1)^{k+1}\frac{p^{k-1}}{k}x^k\right)^m=\sum_{k=m}^{
s-1}S_k^{(m)}\frac{\left(p x\right)^k }{k!}
+\frac{p^{m}}{m!}\frac{x^{s}}{d_{s}^m} Q(x),
\end{equation*}
where $d_{s}$ is the least common multiple of all integers less
than $s$ that are relatively prime to $p$.  Taking $x=\alpha p$
and assuming that $m\ge 1$ and $p>2$, we have
\begin{equation*}
\nu_p\left(\frac{\alpha^{s}p^{s+m}}{m!d_{s}^m} Q\left(\alpha
p\right)\right)\ge \nu_p\left(\frac{p^{s+m}}{m!}\right)>
s+m-\frac{m}{p-1}>s.
\end{equation*}
It follows that
\begin{equation*}
\begin{split}
\sum_{j=m}^{s-1}S_j^{(m)}\frac{\alpha^j p^{2j}}{j!}\equiv
\frac{p^m}{m!}\left(\sum_{j=1}^{s-1}\frac{(-1)^{j+1}}{j}\alpha^j
p^{2j-1}\right)^m\left(\mod p^{s+1}\right),
\end{split}
\end{equation*}
and dividing both sides by $p$ and then simplifying yields
\begin{align}
\sum_{j=m}^{s-1}S_j^{(m)}\frac{\alpha^j p^{2j-1}}{j!} &\equiv
\frac{p^{2m-1}}{m!}z^m\left(\mod p^{s}\right)\label{stirling sum
reduction},
\end{align}
where $z$ is defined in Eq. \eqref{z def}.  Substituting Eq.
\eqref{stirling sum reduction} into Eq. \eqref{first sum
intermediate part two} yields
\begin{equation}
\sum_{j=1}^{n}\frac{y^{p^2 j}-y^{p j}}{p j}\equiv
\sum_{m=0}^{s-2}p^{2m+1}\frac{z^{m+1}}{(m+1)!}\left(\sum_{j=1}^{n}j^m
y^{p j}\right)\left(\mod p^s\right).
\end{equation}
Finally if $y\not= 1$  we can substitute equations
\eqref{summation for finite polylog} and \eqref{truncated polylog
summation formula} to complete the proof. $\blacksquare$
\end{proof}
%
%
%
%

\subsection{Summary of computations} \label{computations} \init

In summary, we calculated at least part of $J_p\left(y^p\right)$ for
every prime $p<550$, and for each integer $y\in\{2,3,\dots,p-1\}$.
We then used the relation $J_p\left(y^p\right)=J_p(y)$ to determine
$J_p(y)$.  We also checked Boyd's calculations of $J_p(1)$ with a
version of our program. In particular, we used Theorem \ref{Gn(y^p)
congruence theorem} to determine $J_p\left((1+p)^p\right)$, and then
we verified Boyd's results from the fact that
$J_p\left((1+p)^p\right)=J_p(1)$.  Table $1$ lists the values of
$|J_p(y)|$ for $p<50$, and an extended list for $p<550$ is available
at \url{www.math.ubc.ca/~matrogers/Papers/padic.html}.

    Some interesting observations follow from our computations.
Firstly, $|J_p(y)|$ is small for many values of $y$ and $p$.   For
example, when we considered the possible values of $|J_p(y)|$ for
$p<50$, we found that only $27$ out of $313$ possible cases have
$|J_p(y)|>50$. The three largest sets for $p<50$ are
$|J_{47}(12)|=40608$, $|J_{47}(8)|=27024$, and $|J_{13}(9)|=18763$.
In particular, we have explicitly proven that the congruence
\begin{equation*}
\sum_{j=1}^{n}\frac{9^j}{j}\equiv 0\left(\mod 13\right),
\end{equation*}
has exactly $18763$ solutions.  The first solution occurs at $n=3$,
while the largest solution has $419$ digits and approximately equals
$n\approx 2.385\times 10^{419}$.

We have also calculated that $|J_p(y)|=0$ in $104$ out of $313$
possible cases for $p<50$. This seems to agree with a simple
heuristic suggesting that the density of such $J_p(y)$'s should
approach $1/e\approx .36$.  To see this fact, notice that
$|J_p(y)|=0$ if and only if $G_n(y)\not\equiv 0 \left(\mod p\right)$
for every $n\in \{1,\dots,p-1\}$.  If we assume that the value of
$G_n(y)\left(\mod p\right)$ is randomly distributed whenever
$y\not=1$ (recall that $|J_p(1)|\ge 3$ for any odd prime $p$), then
it is clear that $|J_p(y)|=0$ with probability
$\left(1-1/p\right)^{p-1}$ when $y\not=1$, and probability zero when
$y=1$. Therefore the expected ratio of empty $J_p(y)$'s for $p<n$
equals
\begin{equation*}
\mathbb{E}\left(|J_p(y)|=0 : p< n, 1\le y\le
p-1\right)=\frac{\sum_{p<n}(p-2)(1-1/p)^{p-1}}{\sum_{p<n}(p-1)}\approx
\frac{1}{e},
\end{equation*}
and the expectation approaches $1/e$ as $n\rightarrow \infty$ by
standard analysis. \setlongtables
\begin{longtable}{|c|c|c|c|c|c|c|c|c|c|c|c|c|c|c|c|}
        \hline
        $y,p$ & 2 & 3 & 5 & 7 & 11 & 13 & 17 & 19 & 23 & 29 & 31 & 37 & 41 & 43 & 47\\
        \hline
        \endhead
        \hline
        \caption*{\textbf{Table $1$ }: Values of $|J_p(y)|$ for $p<50$}\\
        \endfoot
        \caption*{\textbf{Table $1$} \textit{(continued)}: Values of $|J_p(y)|$ for $p<50$}\\
        \endlastfoot
        $1$  & 0 & 3 & 3  & 13  & 638   &  3  & 3   &  19  & 3   & 18   & 26   &  15  &  3  & 27   & 11  \\
        $2$  && 0 & 37 & 0 & 0 & 1 & 0 & 9 & 3 & 2 & 1 & 0 & 29 & 0 & 0\\
        $3$  &&& 4 & 4 & 184&0&4&4&0&0&6&140&0&0&0\\
        $4$  &&& 1& 0& 1& 5& 3& 0& 1& 0& 10& 0& 5& 0& 0\\
        $5$  &&&& 12& 4& 1& 0& 1& 0& 1& 0& 0& 0& 1& 34\\
        $6$  &&&&65& 0& 0& 1& 16& 0& 2& 6& 1& 0& 0& 0\\
        $7$  &&&&&0& 0& 8& 4& 4& 0& 1& 0& 0& 129& 0\\
        $8$  &&&&&0& 0& 4& 5& 1& 6& 0& 325& 7& 0& 27024\\
        $9$  &&&&&26& 18763& 1& 25& 0& 1& 0& 1& 0& 4& 6\\
        $10$ &&&&&1& 2& 6& 1& 0& 1& 1225& 27& 2& 0& 1\\
        $11$ &&&&&& 6& 0& 154& 14& 3& 2& 1& 0& 0& 4\\
        $12$ &&&&&& 11& 45& 13& 0& 3& 0& 0& 17& 0& 40608\\
        $13$ &&&&&&& 0& 1& 2& 0& 0& 4& 0& 1& 0 \\
        $14$ &&&&&&& 1& 0& 1& 2& 1& 133& 3& 13& 349\\
        $15$ &&&&&&& 10& 4& 86& 1& 3& 0& 1& 2& 1\\
        $16$ &&&&&&& 65& 61& 0& 0& 0& 5& 24& 39& 0\\
        $17$ &&&&&&&& 6& 0& 0& 0& 0& 0& 2& 3\\
        $18$ &&&&&&&& 13& 0& 5& 59& 8& 3& 2& 2\\
        $19$ &&&&&&&&& 0& 0& 1& 0& 38& 0& 0\\
        $20$ &&&&&&&&& 1& 1& 1& 151& 6& 5& 0\\
        $21$ &&&&&&&&& 2& 8043& 29& 0& 5& 0& 0\\
        $22$ &&&&&&&&& 1& 0& 0& 0& 13& 0& 1\\
        $23$ &&&&&&&&&& 28& 48& 0& 85& 0& 3\\
        $24$ &&&&&&&&&& 0& 24& 233& 3& 20& 92\\
        $25$ &&&&&&&&&& 0& 0& 4& 0& 0& 0\\
        $26$ &&&&&&&&&& 28& 64& 11& 10& 68& 2\\
        $27$ &&&&&&&&&& 6& 38& 1& 3& 28& 5\\
        $28$ &&&&&&&&&& 8& 0& 0& 2& 0& 0\\
        $29$ &&&&&&&&&&& 2& 3& 14& 8& 3\\
        $30$ &&&&&&&&&&& 4& 1& 0& 0& 8\\
        $31$ &&&&&&&&&&&& 0& 5& 9& 2\\
        $32$ &&&&&&&&&&&& 0& 5743& 18& 0\\
        $33$ &&&&&&&&&&&& 4& 1&  1& 1\\
        $34$ &&&&&&&&&&&& 24& 4&  1& 0\\
        $35$ &&&&&&&&&&&& 6& 0&  1& 0\\
        $36$ &&&&&&&&&&&& 34& 22&  8& 3\\
        $37$ &&&&&&&&&&&&&  4&  14& 1\\
        $38$ &&&&&&&&&&&&&  10&  0& 392\\
        $39$ &&&&&&&&&&&&&  22&  1& 3\\
        $40$ &&&&&&&&&&&&&  32&  1& 5\\
        $41$ &&&&&&&&&&&&&&  8& 21\\
        $42$ &&&&&&&&&&&&&&  10198& 2\\
        $43$ &&&&&&&&&&&&&&& 6\\
        $44$ &&&&&&&&&&&&&&& 1\\
        $45$ &&&&&&&&&&&&&&& 5\\
        $46$ &&&&&&&&&&&&&&& 2\\
       \hline
\end{longtable}
 We can examine cases where $y$ is fixed and $p$ varies in somewhat
greater detail.  Following Boyd we will use the notation
\begin{equation*}
G_m(y)=J_p(y)\bigcap \{p^{m-1},p^{m-1}+1,\dots,p^{m}-1\},
\end{equation*}
to denote a level in the tree $J_p(y)$.  Notice that the integers
contained in $G_m(y)$ are precisely the elements of $J_p(y)$ that
we will discover during the $m$'th iteration of Boyd's algorithm.
As usual, $M_p(y)-1$ equals the number of levels in the tree
$J_p(y)$, notice that
\begin{equation*}
J_p(y)=\bigcup_{m=1}^{M_p(y)-1}G_m(y).
\end{equation*}
As an example we will consider the case that occurs when $y=2$.
The following table lists the nonzero values of $|J_p(2)|$ for
$p<550$.

 \setlongtables
\begin{longtable}{|c|c|c|l|}
        \hline
        $p$ & $M_p(2)$ & $|J_p(2)|$& Values of $|G_m(2)|$ for $1\le m< M_p(2)$\\
        \hline
        \endhead
        \hline
        \caption*{\textbf{Table $2$ }: Primes $p<550$ for which
        $|J_p(2)|>0$. 
        }\\
        \endfoot
        \hline
        \caption*{\textbf{Table $2$} \textit{(continued)} : Primes $p<550$ for which
        $|J_p(2)|>0$.
        }\\
        \endlastfoot
        $5$ & $15$ & $37$ & $1,2,3,4,2,3,3,4,3,2,4,1,2,3$\\
        $13$ & $2$ & $1$ & $1$\\
        $19$ & $5$ & $9$ & $1,2,3,3$\\
        $23$ & $3$ & $3$ & $1,2$\\
        $29$ & $3$ & $2$ & $1,1$\\
        $31$ & $2$ & $1$ & $1$\\
        $41$ & $9$ & $29$ & $1,3,3,5,7,7,2,1$\\
        $53$ & $2$ & $2$ & $2$\\
        $59$ & $5$ & $7$ & $2,2,1,2$\\
        $73$ & $6$ & $11$ & $3,2,1,2,3$\\
        $83$ & $3$ & $2$ & $1,1$\\
        $89$ & $24$ & $56$ & $1,1,1,2,4,4,7,6,4,3,2,2,2,2,1,1,1,3,3,3,1,1,1$\\
        $103$ & $3$ & $3$ & $2,1$\\
        $113$ & $50$ & $394$ &  $1,3,4,7,9,10,10,9,9,7,7,11,8,7,6,7,10,12,15,11,9,12,9,7,7,$\\
        &&&$6,11,8,8,11,14,14,14,11,11,12,9,11,6,7,3,7,6,$
        $3,5,3,2,2,3$\\
        $131$ & $18$ & $80$ & $2,3,4,6,2,3,3,4,3,6,3,6,12,11,7,4,1$\\
        $137$ & $6$ & $9$ & $1,2,2,3,1$\\
        $151$ & $3$ & $3$ & $2,1$\\
        $157$ & $11$ & $18$ & $2,1,1,3,1,2,2,3,2,1$\\
        $163$ & $2$ & $1$ & $1$\\
        $167$ & $2$ & $1$ & $1$\\
        $173$ & $3$ & $3$ & $2,1$\\
        $179$ & $2$ & $1$ & $1$\\
        $181$ & $2$ & $1$ & $1$\\
        $193$ & $4$ & $3$ & $1,1,1$\\
        $197$ & $6$ & $7$ & $2,1,1,1,2$\\
        $199$ & $4$ & $3$ & $1,1,1$\\
        $211$ & $12$ & $41$ & $1,2,5,5,5,6,3,6,4,3,1$\\
        $239$ & $3$ & $7$ & $3,4$\\
        $241$ & $6$ & $9$ & $2,2,2,2,1$\\
        $257$ & $9$ & $17$ & $1,1,1,3,3,2,4,2$\\
        $269$ & $7$ & $7$ & $2,1,1,1,1,1$\\
        $271$ & $6$ & $7$ & $1,3,1,1,1$\\
        $293$ & $2$ & $1$ & $1$\\
        $307$ & $4$ & $6$ & $1,3,1,1$\\
        $311$ & $3$ & $3$ & $1,2$\\
        $313$ & $2$ & $2$ & $2$\\
        $317$ & $4$ & $7$ & $3,3,1$\\
        $331$ & $13$ & $55$ & $1,1,2,3,3,6,10,6,9,9,3,2$\\
        $337$ & $20$ & $47$ & $2, 1, 3, 5, 1, 3, 4, 4, 2, 3, 3, 3, 2, 3, 1, 2, 3, 1, 1$\\
        $349$ & $3$ & $4$ & $1,3$\\
        $367$ & $2$ & $1$ & $1$\\
        $373$ & $4$ & $5$ & $2,2,1$\\
        $379$ & $6$ & $19$ & $4,3,6,4,2$\\
        $383$ & $3$ & $2$ & $1,1$\\
        $389$ & $17$ & $51$ &  $1, 2, 2, 2, 2, 5, 7, 8, 6, 6, 3, 2, 1, 1, 1, 2$\\
        $397$ & $13$ & $33$ & $1, 3, 3, 4, 3, 1, 3, 2, 6, 3, 3, 1$\\
        $401$ & $2$ & $1$ & $1$\\
        $419$ & $2$ & $1$ & $1$\\
        $431$ & $12$ & $76$ & $3, 6, 4, 10, 11, 10, 8, 8, 9, 4, 3$\\
        $439$ & $14$ & $26$ & $1, 1, 1, 1, 1, 1, 3, 4, 3, 3, 2, 3, 2$\\
        $449$ & $3$ & $2$ & $1,1$\\
        $457$ & $7$ & $7$ & $1,1,2,1,1,1$\\
        $461$ & $2$ & $1$ & $1$\\
        $463$ & $8$ & $12$ & $1,2,2,1,3,2,1$\\
        $479$ & $2$ & $2$ & $2$\\
        $487$ & $21$ & $52$ & $2, 2, 1, 3, 3, 4, 6, 2, 3, 4, 2, 3, 2, 2, 1, 3, 4, 2, 2, 1$\\
        $499$ & $30$ & $272$ & $1, 6, 6, 6, 6, 11, 9, 11, 10,
 16, 15, 18, 14, 18, 16, 11,$\\
&&&$10, 11, 8, 6, 9, 9, 10, 8,
 7, 5, 8, 5, 2$\\
        $509$ & $5$ & $4$ & $1,1,1,1$\\
        $523$ & $8$ & $16$ & $2,2,2,2,4,3,1$\\
        $547$ & $4$ & $4$ & $1,2,1$\\
        \hline
\end{longtable}
 Finally, we will point out that it is usually easy to calculate $J_{p^s}\left(y^p\right)$
after first determining $J_p(y)$.  Since we will have already
calculated the value of $G_n\left(y^p\right)\left(\mod
p^{s'}\right)$ for some $s'\gg 1$, we can simply check whether or
not $G_n\left(y^p\right)\equiv 0\left(\mod p^j\right)$ for every
$n\in J_p(y)$ and for any $j<s'$. In practice this check rarely
requires new computations, since generally we will have used a value
of $s'$ much larger than the order of vanishing of
$G_n\left(y^p\right)$ modulo $p$. As an example we proved that
\begin{align*}
 J_5(2)=\{&{3}, {17, 19}, {86, 97, 99}, {485, 488, 497, 499},
2486, 2496,12431, \\
& 12482, 12484, 62157, 62159, 62421,
 {310787, 310789, 312107},\\
 &312109,{1553936, 1560537, 1560539},
 {7802685, 7802688},\\
 & {39013425, 39013428, 39013442, 39013444}, 195067126,\\
 &{975335630, 975335633}, {4876678152, 4876678154, 4876678166}\},
\end{align*}
and with minimal extra computations we also determined that
\begin{align*}
J_{25}(7)=&\{ 3,19,499,2486,12431,312107\},\\
J_{125}(32)=&\emptyset.
\end{align*}
Notice that $J_{5^{s}}(32)=\emptyset$ when $s\ge 3$, since in those
cases $J_{5^s}(32)\subset J_{125}(32)=\emptyset$.

Unfortunately this procedure is unsuitable for calculating the
majority of values of $J_{p^{s}}(x)$ when $s>1$. Although we can
easily calculate $J_{p^s}\left(y^p\right)$, the method only applies
to $J_{p^{s}}(x)$ when $x\equiv y^p \left(\mod p^s \right)$ for some
$y$.  For example, if $x\in\mathbb{Z}_{25}^{*}\setminus
\{1,7,18,24\}$, then we have to settle for the weak conclusion that
$J_{25}(x)\subset J_{5}(x)$. While Theorem \ref{Gn(y^p) congruence
theorem} probably only requires minor modifications to extend the
computations, we will not address that problem here.

\medskip
\noindent\textbf{Open Problem :} Calculate $J_{p^{s}}(x)$ when $s>1$
and $x\not\equiv y^p\left(\mod p^s\right)$ for any $y$.
\medskip

\section{Conclusion}
\label{conclusion} \init
Although we primarily restricted our attention to computational
problems in this paper, it would be desirable to construct
probabilistic models to explain the behavior of $|J_p(y)|$.  Boyd
constructed such models to explain the behavior of $|J_p(1)|$, and
it seems likely that his ideas will suffice to explain our results
as well. Notice that after combining our calculations with Boyd's,
we have proved that $|J_p(y)|<\infty$ in $98.6 \%$ of cases for
primes $p<550$. In fact, it would be desirable to find a general
proof of the following conjecture:

\medskip
\noindent\textbf{Conjecture 1 :} We will conjecture that
$|J_p(y)|<\infty$ whenever $\nu_p(y)=0$.
\medskip

Finally, we will conclude the paper with an example of a more
complicated type of function that we can easily study with our
method. Consider the function $f_n$ defined by
\begin{equation*}
f_n:=\sum_{j=1}^{n}\frac{F_j}{j},
\end{equation*}
where the $F_j$'s are Fibonacci numbers.  Recall that we can
either calculate $F_j$ recursively, or with Binet's formula:
\begin{equation*}
F_j=\frac{\left(1+\sqrt{5}\right)^j-\left(1-\sqrt{5}\right)^j}{2^j
\sqrt{5}}.
\end{equation*}

    It is not difficult to prove that the solutions of $f_n\equiv 0\left(\mod
p\right)$ are arranged in a tree.  The crucial fact for proving
this claim is that the Fibonacci numbers satisfy the congruence
\begin{equation}\label{fibonacci congruence}
F_{p^s j}\equiv \left(\frac{p}{5}\right)F_{p^{s-1}j} \left(\mod
p^s\right),
\end{equation}
for all $s$ and $j$, with $\left(\frac{*}{*}\right)$ denoting the
Legendre symbol.  This congruence easily implies that
$f^{(1)}_n\equiv \left(\frac{p}{5}\right)f_n\left(\mod p\right)$,
where
\begin{equation*}
f^{(1)}_n=\sum_{j=1}^{n}\frac{F_{p j}}{j}.
\end{equation*}
Using properties of the Fibonacci numbers we can prove that
$f^{(1)}_n$ obeys the congruence
\begin{equation}\label{p adic fibonacci recurrence}
f^{(1)}_{p n+j}\equiv
\left(\frac{p}{5}\right)\frac{f^{(1)}_n}{p}+\left(\frac{p}{5}\right)^2
F_{n}\sum_{k=1}^{j}\frac{F_{k+1}}{k}+\left(\frac{p}{5}\right)^2\left(F_{n+1}-1\right)\sum_{k=1}^{p-1}\frac{F_{k+1}}{k}
\left(\mod p\right),
\end{equation}
and therefore our claim about the distribution of zeros of $f_n
\left(\mod p\right)$ follows immediately.  Based on cursory
computations it also seems reasonable to make the following
conjecture:

\medskip
\noindent\textbf{Conjecture 2 :} For all $n$ we have $f_{4n}\equiv 0
\left(\mod 5\right)$.  Furthermore, if $p\not =5$:
\begin{equation*}
\left|\left\{n : f_n\equiv 0 \left(\mod
p\right)\right\}\right|<\infty.
\end{equation*}

From these short computations, it seems obvious that our method will
extend to functions involving integer sequences other than just the
Fibonacci numbers.  We will speculate that many functions of the
form
\begin{equation*}
\sum_{j=1}^{n}\frac{T_j}{j}
\end{equation*}
should obey $p$-adic recurrences analogous to equations
\eqref{Gn(y^p) congruence theorem} or \eqref{p adic fibonacci
recurrence}, provided that $T_j$ satisfies a second-degree linear
recurrence. We are not presently prepared to speculate on the
behavior of such functions when the $T_j$'s satisfy higher order
recurrences, as we failed to observe any interesting patterns modulo
$p$ when the $T_j$'s equal Tribonacci numbers.


\section{Appendix : A simple but important sum}
\label{appendix} \init

Virtually all of the calculations in this paper depended upon our
ability to efficiently calculate the simple sum
\begin{equation*}
\sum_{j=1}^{n}j^r x^j
\end{equation*}
for extremely large, \textit{but finite}, values of $n$.  It was
therefore imperative to eliminate the $n$-dependency from the
index of summation. While many obvious formulas exist for this
sum, including
\begin{equation*}
\sum_{j=1}^{n}j^r x^j=\left(x\frac{\d}{\d x}\right)^r
\left(\frac{x(1-x^{n})}{1-x}\right),
\end{equation*}
we chose to avoid recursive identities, and to find a closed form
instead. Perhaps the crucial observation was that when $n=\infty$:
\begin{equation}\label{summation for finite polylog}
\Li_{-r}(x):=\sum_{j=1}^{\infty}j^r
x^j=-\delta_{r0}+\sum_{j=0}^{r}j!\mathfrak{S}_{r}^{(j)}\frac{x^j}{(1-x)^{j+1}}.
\end{equation}
As usual $\delta_{r0}$ is the Kronecker delta, and
$\mathfrak{S}_r^{(j)}$ denotes the Stirling numbers of the second
kind \cite{Gr}. Therefore, briefly assuming that $|x|<1$, we obtain
\begin{equation*}
\sum_{j=1}^{n}j^r x^j=\sum_{j=1}^{\infty}j^r x^j
-\sum_{j=1}^{\infty}(n+j)^r x^{n+j},
\end{equation*}
and expanding $(n+j)^r$ with the binomial formula yields
\begin{equation}\label{truncated polylog summation formula}
\sum_{j=1}^{n}j^r x^j=\Li_{-r}(x)-x^n\sum_{m=0}^{r}{r\choose m}n^m
\Li_{-(r-m)}(x).
\end{equation}
Since the left-hand side of this last identity is a polynomial, the
principle of analytic continuation shows that the identity holds for
all $x$. In practice we should only use Eq. \eqref{truncated polylog
summation formula} if $x\not =1$, since in that case we can use Eq.
\eqref{summation for finite polylog} to calculate $\Li_{-r}(x)$.
When $x=1$ we can calculate the left-hand side of Eq.
\eqref{truncated polylog summation formula} by simply reverting to
Bernoulli's classical formula for power sums.
%
%
\bigskip

\begin{acknowledgements}

The author would like to thank David Boyd for the many useful
discussions, and for his kind encouragement.
\end{acknowledgements}


\begin{thebibliography}{aaa}

\bibitem{Bo} David W. Boyd : {\it A $p$-adic Study of the Partial Sums of the Harmonic Series},
Experimental Mathematics, Volume 3, Number 4, 1994
%
%
\bibitem{Es} Jody Esmonde and M. Ram Murty : {\it Problems in Algebraic Number
Theory}, Springer-Verlag, 1999

\bibitem{Gr} I.S. Gradshteyn and I.M. Ryzhik : {\it Table
of Integrals, Series and Products}, Academic Press 1994
%
%

\bibitem{Ri1} Paulo Ribenboim : {\it 13 Lectures on Fermat's Last
Theorem}, Springer-Verlag, 1979
%
%
%
%

\end{thebibliography}
\end{document}